\documentclass[a4paper,11pt]{amsart}
\usepackage[T1]{fontenc}
\usepackage{lmodern}

\usepackage{amsmath,amssymb,amsthm,mathtools}
\usepackage{needspace}
\usepackage[margin=25mm]{geometry}
\usepackage[colorlinks=true,linkcolor=blue,citecolor=blue,urlcolor=blue]{hyperref}
\numberwithin{equation}{section}
\newcommand{\N}{\mathbb N}
\newcommand{\Z}{\mathbb Z}
\newcommand{\R}{\mathbb R}
\newcommand{\E}{\mathbb E}
\newcommand{\Prob}{\mathbb P}
\newcommand{\Kc}{K_{\mathrm{CBB}(0)}}
\newcommand{\mc}{m_{\mathrm{CBB}(0)}}
\date{}

\newtheorem{lemma}{Lemma}[section]
\newtheorem{theorem}[lemma]{Theorem}
\theoremstyle{definition}
\newtheorem{definition}[lemma]{Definition}
\theoremstyle{remark}
\newtheorem{remark}[lemma]{Remark}
\newtheorem{question}[lemma]{Question}

\title[The linear scale in metric cotype]{The linear scale in metric cotype\\for Alexandrov spaces of nonnegative curvature}

\author{Longfa Sun}
\address{Hebei Key Laboratory of Physics and Energy Technology, School of Mathematics and Physics, North China Electric Power University, Baoding, 071003, China}
\email{sun.longfa@ncepu.edu.cn}

\author{Yipeng Zhang}
\address{School of Mathematical Sciences, Xiamen University, Xiamen 361005, China}
\email{19020250157654@stu.xmu.edu.cn}

\makeatletter
\renewcommand{\author@andify}{%
  \nxandlist{\unskip,\,\,}{\unskip,\,\,}{\unskip,\,\,}%
}
\makeatother
\begin{document}
\begin{abstract}
We answer a question of Eskenazis, Mendel, and Naor by proving that the optimal universal scaling parameter in metric cotype $2$ for complete Alexandrov spaces of nonnegative curvature has order $n$. The upper bound follows from a dyadic application of the Lang--Schroeder inequality, and the lower bound follows from snowflake universality of $\mathcal P_2(\R^3)$.
\end{abstract}
\maketitle
\section{Introduction}
We determine the optimal scaling parameter in the metric cotype $2$ inequality for Alexandrov spaces of nonnegative curvature.
This inequality compares long coordinate increments of a map on a discrete torus with short random increments.
The scaling parameter $m$ governs this comparison, and its dependence on the dimension $n$ is important in applications.
Eskenazis, Mendel, and Naor~\cite[Question 17]{EMN} asked for the optimal order of $m$ for this class of spaces.
We prove that, for every fixed constant $\gamma>1$, the smallest scale that works uniformly over all complete Alexandrov spaces of nonnegative curvature has order $n$.

This result reveals a difference between the two signs of Alexandrov curvature.
Eskenazis, Mendel, and Naor~\cite{EMN} proved that the optimal universal scale for complete Alexandrov spaces of nonpositive curvature has order $\sqrt n$.
Thus both classes have metric cotype $2$, but nonnegative curvature requires a larger universal scale.

We first recall the classical Rademacher cotype.
A Banach space $(X,\|\cdot\|_X)$ has \emph{Rademacher cotype $q$}, where $q\in[2,\infty)$, if there is $C<\infty$ such that
\[
\sum_{j=1}^n\|x_j\|_X^q\le C^q\E_\varepsilon\Big\|\sum_{j=1}^n\varepsilon_jx_j\Big\|_X^q
\]
for every $n\in\N$ and $x_1,\ldots,x_n\in X$, with $\varepsilon$ uniform on $\{-1,1\}^n$.
Mendel and Naor~\cite{MN} introduced metric cotype and proved in 2008 that it characterizes Rademacher cotype for Banach spaces. Thus cotype admits a formulation involving only distances.

We use the definition of metric cotype from~\cite{EMN}, in which the random increment $\varepsilon$ is chosen uniformly from $\{-1,1\}^n$ rather than $\{-1,0,1\}^n$. The quantitative relationship between these two formulations is explained in~\cite[Section 4]{EMN}. Write $\Z_{2m}=\Z/(2m\Z)$ and let $e_1,\ldots,e_n$ be the standard coordinate vectors. \begin{definition}
A metric space $(X,d_X)$ has \emph{metric cotype $q$} with constant $\gamma$ if, for every $n\in\N$, there is an even integer $m\ge2$ such that every $f:\Z_{2m}^{\,n}\to X$ satisfies
\begin{equation}\label{eq:cotype}
\sum_{j=1}^n\E_x d_X(f(x+me_j),f(x))^q
\le \gamma^qm^q\E_{x,\varepsilon}d_X(f(x+\varepsilon),f(x))^q,
\end{equation}
where $x$ and $\varepsilon$ are independent and uniform on $\Z_{2m}^{\,n}$ and $\{-1,1\}^n$, respectively. \end{definition}

We next recall the definitions of Alexandrov spaces of nonnegative and nonpositive curvature.
\begin{definition}
A metric space is \emph{geodesic} if every pair $x,y$ is joined by a map $g:[0,1]\to X$ with $d_X(g(s),g(t))=|s-t|d_X(x,y)$. A complete geodesic space is \emph{CBB$(0)$}, or an Alexandrov space of nonnegative curvature, if every $x,y,z\in X$ and every midpoint $w$ of $x,y$ satisfy
\begin{equation}\label{eq:curvature}
d_X(z,w)^2+\frac14d_X(x,y)^2
\ge\frac12d_X(z,x)^2+\frac12d_X(z,y)^2.
\end{equation}
Reversing this inequality gives the definition of a \emph{CAT$(0)$} space. \end{definition}
We follow~\cite{AKP} for the notation CBB$(0)$ and~\cite{BH} for the notation CAT$(0)$. Hilbert spaces belong to both classes, while intrinsic circles are CBB$(0)$ and metric trees are CAT$(0)$.

A metric space has \emph{sharp metric cotype $q$} if there exist constants $\gamma,C<\infty$ such that, for every $n$, the inequality holds at some even scale $m\le Cn^{1/q}$.
Eskenazis, Mendel, and Naor~\cite{EMN} proved that complete CAT$(0)$ spaces have sharp metric cotype $2$ with universal constants. Together with the general lower bound
from~\cite[Lemma 2.3]{MN}, their result shows that the optimal universal scale satisfies $m\asymp\sqrt n$.

For CBB$(0)$ spaces, they announced metric cotype $2$ in~\cite[Remark 16]{EMN}, but the optimal scale remained undetermined. They also observed that $\mathcal P_2(\R^3)$ has no sharp metric cotype. They therefore posed the following problem, stated here in the normalization~\eqref{eq:cotype}.
\begin{question}[\cite{EMN}, Question 17]
What is the asymptotic order of the smallest universal scale $m=m(n)$ for which every complete CBB$(0)$ space satisfies~\eqref{eq:cotype} with $q=2$ and a universal constant $\gamma$?
\end{question}

Our answer is the following.
\begin{theorem}\label{thm:main}
Complete CBB$(0)$ spaces have metric cotype $2$ with a universal constant and scale $m\asymp n$. This order is optimal.
\end{theorem}
Thus the two signs of curvature give different optimal scales:
\[
\mathrm{CAT}(0):\quad m\asymp\sqrt n,
\qquad
\mathrm{CBB}(0):\quad m\asymp n.
\]

The proof has two parts. The upper bound follows from a dyadic
application of the Lang--Schroeder inequality. The lower bound
uses the snowflake universality theorem of Andoni, Naor, and
Neiman~\cite{ANN} to construct maps into $\mathcal P_2(\R^3)$
that force $m\gtrsim n$.
\section{Proof of Theorem~\ref{thm:main}}
We prove a stronger estimate that gives the dependence on both $n$ and $m$ explicitly.
\begin{definition}\label{def:K}
For $n\in\N$ and $m\in2\N$, define
\begin{equation}\label{eq:K}
\Kc(n,m)=
\sup_{\substack{(X,d_X)\text{ is complete CBB}(0)\\
f:\Z_{2m}^{\,n}\to X}}
\frac{\displaystyle
\sum_{j=1}^n\E_xd_X(f(x+me_j),f(x))^2}
{\displaystyle
m^2\E_{x,\varepsilon}d_X(f(x+\varepsilon),f(x))^2}.
\end{equation}
Here we use the conventions $0/0=0$.
\end{definition}
We will show that, for dyadic $m\ge2$,
\begin{equation}\label{eq:range}
\max\{1,n/m\}\le\Kc(n,m)\le e\lceil n/m\rceil.
\end{equation}

\subsection{The upper bound for \texorpdfstring{$\Kc$}{K-CBB(0)}}

The proof in this subsection uses the result of Lang and
Schroeder~\cite{LS}: if $Z$ is a finitely supported random variable in a complete CBB$(0)$ space and $Z'$ is an independent copy, then
\begin{equation}\label{eq:lss}
\E d_X(Z,Z')^2\le2\inf_{z\in X}\E d_X(Z,z)^2.
\end{equation}
We first put this inequality in a translation-invariant form.
\begin{lemma}\label{lem:lss}
Let $X$ be complete and CBB$(0)$, and let $f:\Z_{2m}^{\,n}\to X$. If $U,U'$ are independent with the same distribution on $\Z_{2m}^{\,n}$ and independent of uniform $x$, then
\[
\E_{x,U,U'}d_X(f(x+U-U'),f(x))^2
\le2\E_{x,U}d_X(f(x+U),f(x))^2.
\]
\end{lemma}
\begin{proof}
For each $y\in\Z_{2m}^{\,n}$, by~\eqref{eq:lss}, we have
\[
\E_{U,U'}d_X(f(y+U),f(y+U'))^2
\le2\E_Ud_X(f(y+U),f(y))^2.
\]
Averaging over $y$ and using the change of variable $x=y+U'$,
we obtain
\[
\begin{aligned}
\E_{x,U,U'}d_X(f(x+U-U'),f(x))^2
&=
\E_{y,U,U'}d_X(f(y+U),f(y+U'))^2\\
&\le
2\E_{y,U}d_X(f(y+U),f(y))^2.
\end{aligned}
\]
This proves the lemma.
\end{proof}

The proof of the next lemma is inspired by Ohta's dyadic
argument~\cite[Lemma 4.1]{Ohta}. We construct random increments
$V_s$ by halving the retention probability while doubling the
nonzero increments at each step. These random increments form
a bridge between the short sign increments $\varepsilon$
and the long increments $me_j$. Comparing successive stages then yields the desired inequality.

For set $B\subseteq\{1,\ldots,n\}$, write
\[
B^c=\{1,\ldots,n\}\setminus B,
\qquad |B|=\#B,
\qquad
\{-1,1\}^B=\{(\sigma_j)_{j\in B}:\sigma_j\in\{-1,1\}\}.
\]

For random variables $Y$ and $Z$, we write
$\E_{Y\mid Z}$ for expectation with respect to the conditional
distribution of $Y$ given $Z$.
\begin{lemma}\label{lem:block}
Let $m=2^r$ for some $r\in\N$, and let $B$ be a nonempty
subset of $\{1,\ldots,n\}$. Every map $f:\Z_{2m}^{\,n}\to X$ into a complete CBB$(0)$ space satisfies
\[
\sum_{j\in B}\E_xd_X(f(x+me_j),f(x))^2
\le m^2\left(1-\frac1m\right)^{1-|B|}
\E_{x,\varepsilon}d_X(f(x+\varepsilon),f(x))^2.
\]
\end{lemma}
\begin{proof}
To connect the short sign increments $\varepsilon$ to the long increments $me_j$, for $1\le s\le r$ let $V_s$ have independent coordinates, zero off $B$, and on $B$ let
\begin{equation*}\label{eq:law}
\Prob((V_s)_j=0)=1-2^{-s},\qquad
\Prob((V_s)_j=2^s)=\Prob((V_s)_j=-2^s)=2^{-s-1}.
\end{equation*}
When $s=r$, we have $2^r=m=-m$ in $\Z_{2m}$, so the probabilities of the two signed outcomes add up to $1/m$.
Next we divide the proof into three steps.

\textbf{Step 1. From $\varepsilon$ to $V_1$.}
Fix $\eta\in\{-1,1\}^{B^c}$ and let $\sigma,\sigma'$ be independent uniform signs on $B$. Set
\[
U_j=\begin{cases}\sigma_j,&j\in B,\\\eta_j,&j\notin B,\end{cases}
\qquad
U'_j=\begin{cases}\sigma'_j,&j\in B,\\\eta_j,&j\notin B.\end{cases}
\]
Then their difference has independent coordinates on $B$ with probabilities
\[
\Prob(U_j-U'_j=0)=\frac12,\qquad
\Prob(U_j-U'_j=2)=\Prob(U_j-U'_j=-2)=\frac14,
\]
and vanishes off $B$. Note that $U-U'$ has the same distribution as $V_1$, by Lemma~\ref{lem:lss}, we have
\begin{equation}\label{eq:first}
\E_{x,V_1}d_X(f(x+V_1),f(x))^2
\le 2\E_\eta\E_{x,U\mid \eta}d_X(f(x+U),f(x))^2
=2\E_{x,\varepsilon}d_X(f(x+\varepsilon),f(x))^2.
\end{equation}

\textbf{Step 2. From $V_s$ to $V_{s+1}$.}
  Fix $s<r$. Choose a random subset $S\subseteq B$ by including each
$j\in B$ independently with probability $2^{-s}$.
For each fixed $S$, let
\[
U_j=
\begin{cases}
2^s\sigma_j,&j\in S,\\
0,&j\notin S,
\end{cases}
\qquad
U'_j=
\begin{cases}
2^s\sigma'_j,&j\in S,\\
0,&j\notin S,
\end{cases}
\]
where all the signs $\sigma_j,\sigma'_j$ are independent
and take the values $-1$ and $1$ with equal probability. Note that
\[
\Prob(U_j-U'_j=0)=1-2^{-s-1},
\qquad
\Prob(U_j-U'_j=2^{s+1})=\Prob(U_j-U'_j=-2^{s+1})=2^{-s-2}
\]
for $j\in B$. Thus $U-U'$ has the same distribution as $V_{s+1}$. Applying Lemma~\ref{lem:lss} for fixed $S$ gives
\begin{equation}\label{eq:stage}
\begin{aligned}
\E_{x,V_{s+1}}d_X(f(x+V_{s+1}),f(x))^2
&=\E_S\E_{x,U,U'\mid S}d_X(f(x+U-U'),f(x))^2\\
&\le2\E_S\E_{x,U\mid S}d_X(f(x+U),f(x))^2\\
&=2\E_{x,V_s}d_X(f(x+V_s),f(x))^2.
\end{aligned}
\end{equation}
Combining~\eqref{eq:first} and~\eqref{eq:stage}, we obtain
\begin{equation}\label{eq:iterate}
\E_{x,V_r}d_X(f(x+V_r),f(x))^2
\le m\E_{x,\varepsilon}d_X(f(x+\varepsilon),f(x))^2.
\end{equation}

\textbf{Step 3. From $V_r$ to $me_j$.}
Since $m=-m$ in $\Z_{2m}$, for every $A\subseteq B$, the final distribution satisfies
\[
\Prob\left(V_r=m\sum_{j\in A}e_j\right)
=\left(\frac1m\right)^{|A|}\left(1-\frac1m\right)^{|B|-|A|}.
\]
Let $A=\{j\}$, we obtain
\[
\begin{aligned}
\E_{x,V_r}d_X(f(x+V_r),f(x))^2
&=\sum_{A\subseteq B}\left(\frac1m\right)^{|A|}
\left(1-\frac1m\right)^{|B|-|A|}
\E_xd_X\left(f\left(x+m\sum_{j\in A}e_j\right),f(x)\right)^2\\
&\ge\frac1m\left(1-\frac1m\right)^{|B|-1}
\sum_{j\in B}\E_xd_X(f(x+me_j),f(x))^2.
\end{aligned}
\]
Combining this with~\eqref{eq:iterate} proves the claim.
\end{proof}

To obtain a bound for all coordinates, we partition them into blocks of size at most $m$. On each block $B$, the factor $\left(1-\frac1m\right)^{1-|B|}$ in Lemma~\ref{lem:block} is bounded by $e$.
\begin{lemma}\label{lem:upper}
For $n\in\N$ and dyadic $m\ge2$, we have $\Kc(n,m)\le e\lceil n/m\rceil$.
\end{lemma}
\begin{proof}
Partition $\{1,\ldots,n\}$ into $\lceil n/m\rceil$ nonempty blocks of size at most $m$. For each such block $B$,
\[
\left(1-\frac1m\right)^{1-|B|}
\le\left(1+\frac1{m-1}\right)^{m-1}<e.
\]
Hence Lemma~\ref{lem:block} gives
\[
\sum_{j=1}^n\E_xd_X(f(x+me_j),f(x))^2
\le e\left\lceil\frac nm\right\rceil m^2
\E_{x,\varepsilon}d_X(f(x+\varepsilon),f(x))^2.
\]
Taking the supremum in~\eqref{eq:K} proves the lemma.
\end{proof}

\subsection{The lower bound for \texorpdfstring{$\Kc$}{K-CBB(0)}}
For the lower bound, we want the long increments to be much larger than the sign increments. The $\ell_\infty$ product of the cyclic distance has exactly this property. We realize its square-root snowflake in a fixed CBB$(0)$ space using the theorem of Andoni, Naor, and Neiman~\cite[Theorem 1]{ANN}.

Recall that $\mathcal P_2(\R^3)$ consists of the Borel probability measures $\mu$ on $\R^3$ with $\int|u|_2^2\,d\mu(u)<\infty$. Its $2$-Wasserstein distance is
\[
W_2(\mu,\nu)^2=\inf_{\pi\in\Pi(\mu,\nu)}\int|u-v|_2^2\,d\pi(u,v),
\]
where $\Pi(\mu,\nu)$ is the set of probability measures on $\R^3\times\R^3$ with marginals $\mu,\nu$. This space is complete and CBB$(0)$, see~\cite[Section 1.3]{ANN}. The same paper proves that, for every finite metric space $(M,\rho)$ and every $\delta>0$, $(M,\sqrt\rho)$ embeds into $\mathcal P_2(\R^3)$ with distortion at most $1+\delta$.

\begin{lemma}\label{lem:lower}
For $n\in\N$ and $m\in2\N$, one has $\Kc(n,m)\ge n/m$.
\end{lemma}
\begin{proof}
Equip $\Z_{2m}^{\,n}$ with the metric
\[
\rho(x,y)=\max_{1\le j\le n}\min_{a\in\Z}|x_j-y_j+2ma|.
\]
Then $\rho(x+me_j,x)=m$ and $\rho(x+\varepsilon,x)=1$. By snowflake universality, for every $\delta>0$ there are $a>0$ and $F:\Z_{2m}^{\,n}\to\mathcal P_2(\R^3)$ such that
\[
a\sqrt{\rho(x,y)}\le W_2(F(x),F(y))\le(1+\delta)a\sqrt{\rho(x,y)}.
\]
It follows that
\[
\sum_{j=1}^n\E_xW_2(F(x+me_j),F(x))^2\ge na^2m,
\qquad
\E_{x,\varepsilon}W_2(F(x+\varepsilon),F(x))^2\le(1+\delta)^2a^2.
\]
Therefore
\[
\Kc(n,m)\ge\frac{n}{m(1+\delta)^2}.
\]
Since $\delta>0$ is arbitrary, we conclude that
$\Kc(n,m)\ge n/m$.
\end{proof}

We next prove that $\Kc(n,m)\ge1$ for every $n\in\N$
and $m\in2\N$. Consequently, no constant $\gamma<1$
can make the metric cotype $2$ inequality hold uniformly
over all complete CBB$(0)$ spaces.
\begin{lemma}\label{lem:one}
For $n\in\N$ and $m\in2\N$, one has $\Kc(n,m)\ge1$.
\end{lemma}
\begin{proof}
Let $C_{2m}=\R/(2m\Z)$ be the circle equipped with the distance
\[
d_{C_{2m}}([u],[v])=\min_{a\in\Z}|u-v+2ma|.
\]
This space is complete and CBB$(0)$. Define
\[
f:\Z_{2m}^{\,n}\to C_{2m},\qquad f(x)=[x_1].
\]
Then
\[
d_{C_{2m}}(f(x+me_j),f(x))
=
\begin{cases}
m,&j=1,\\
0,&j\ne1,
\end{cases}
\qquad
d_{C_{2m}}(f(x+\varepsilon),f(x))=1.
\]
Therefore,
\[
\Kc(n,m)\ge
\frac{\displaystyle
\sum_{j=1}^n\E_xd_{C_{2m}}(f(x+me_j),f(x))^2}
{\displaystyle
m^2\E_{x,\varepsilon}d_{C_{2m}}(f(x+\varepsilon),f(x))^2}
=\frac{m^2}{m^2\cdot1}=1.
\]
\end{proof}

Lemmas~\ref{lem:upper}, \ref{lem:lower}, and~\ref{lem:one} give~\eqref{eq:range}. In particular,
\[
\Kc(n,m)\asymp1+\frac nm\qquad(m=2^r, r\in \mathbb N).
\]
The lower bounds hold for every even $m$, the upper bounds hold for every $m=2^r\, (r\in \mathbb N)$.

\subsection{The main theorem}\label{sec:scale}
To state the main theorem precisely, we use the following definition.
\begin{definition}\label{def:scale}
For $n\in\N$ and $\gamma>0$, define
\begin{equation}\label{eq:scale-def}
\mc(n,\gamma)=\inf\{m\in2\N:\Kc(n,m)\le\gamma^2\},
\qquad \inf\varnothing=+\infty.
\end{equation}
\end{definition}
This definition is exactly the least even scale for which~\eqref{eq:cotype}, with $q=2$, holds for every complete CBB$(0)$ space. Indeed, $\Kc(n,m)\le\gamma^2$ if and only if
every map into a complete CBB$(0)$ space satisfies
\eqref{eq:cotype} with $q=2$ at scale $m$. Thus estimating $\mc$ is equivalent to estimating the optimal universal scale in metric cotype.

We now use the estimates for $\Kc(n,m)$ to show that,
for every fixed $\gamma>1$, the optimal scale
$\mc(n,\gamma)$ has linear order in $n$.
\begin{theorem}\label{thm:scale}
For every fixed $\gamma>1$,
\[
\mc(n,\gamma)\asymp n.
\]
\end{theorem}

\begin{proof}
By Lemma~\ref{lem:lower}, every even scale $m$ satisfying
$\Kc(n,m)\le\gamma^2$ must satisfy
\[
\frac nm\le\Kc(n,m)\le\gamma^2.
\]
Therefore,
\[
\mc(n,\gamma)\ge\frac n{\gamma^2}.
\]

For the upper bound, apply Lemma~\ref{lem:block} with
$B=\{1,\ldots,n\}$. For every $m=2^r$, $r\in\N$, we obtain
\[
\Kc(n,m)
\le\left(1-\frac1m\right)^{1-n}
\le\exp\left(\frac{n-1}{m-1}\right),
\]
where we used $-\log(1-1/m)\le1/(m-1)$.
Choose $m=2^r$, where $r\in\N$ is the smallest integer
satisfying
\[
m\ge\max\left\{2,\,
1+\frac{n-1}{2\log\gamma}\right\}.
\]
Then
\[
\Kc(n,m)\le\gamma^2,
\qquad
m<2\max\left\{2,\,
1+\frac{n-1}{2\log\gamma}\right\}.
\]
Hence $\mc(n,\gamma)\lesssim n$. Together with the lower bound, this proves the theorem.
\end{proof}
The following two remarks explain why we require $m$ to be even
in the definition of $\Kc(n,m)$. When $n\ge2$, odd scales
cannot satisfy the inequality with any finite constant.
For even scales, a finite constant exists for every metric space.

\begin{remark}
Let $2\leq n\in\N$ and $m$ be odd. The map $f:\Z_{2m}^{\,n}\to\R$ defined by $f(x)=(x_1+x_2)\bmod2$. Then for every $x\in\Z_{2m}^{\,n}$ and $\varepsilon\in\{-1,1\}^n$,
\begin{align*}
|f(x+\varepsilon)-f(x)|
&=0,\\
|f(x+me_j)-f(x)|
&=\begin{cases}1,&j\in\{1,2\},\\0,&3\le j\le n.\end{cases}
\end{align*}
Consequently, no finite constant can satisfy~\eqref{eq:cotype} since $\R$ is complete and CBB$(0)$. This is why odd $m$ are excluded when $n\ge2$.
\end{remark}

\begin{remark}
Let $n\in\N$ and $m$ be even, and let
$f:\Z_{2m}^{\,n}\to X$, where $(X,d_X)$ is any metric space.
Fix $j\in\{1,\ldots,n\}$. For
$\varepsilon\in\{-1,1\}^n$, define
\[
\varepsilon'_k=
\begin{cases}
\varepsilon_j,&k=j,\\
-\varepsilon_k,&k\ne j.
\end{cases}
\]
Set
\[
\begin{aligned}
v_{2a}&=a\varepsilon+a\varepsilon',
&&0\le a\le m/2,\\
v_{2a+1}&=(a+1)\varepsilon+a\varepsilon',
&&0\le a<m/2.
\end{aligned}
\]

By the triangle inequality and Cauchy--Schwarz inequality,
\[
\begin{aligned}
d_X(f(x+me_j),f(x))^2
&\le\left(\sum_{t=1}^m
d_X(f(x+v_t),f(x+v_{t-1}))\right)^2\\
&\le m\sum_{t=1}^m
d_X(f(x+v_t),f(x+v_{t-1}))^2.
\end{aligned}
\]

For each $t$, translation invariance in $x$ gives
\[
\E_{x,\varepsilon}
d_X(f(x+v_t),f(x+v_{t-1}))^2
=
\E_{x,\varepsilon}
d_X(f(x+\varepsilon),f(x))^2,
\]
since both $\varepsilon$ and $\varepsilon'$ are uniform
on $\{-1,1\}^n$. Consequently,
\[
\sum_{j=1}^n\E_xd_X(f(x+me_j),f(x))^2
\le nm^2\E_{x,\varepsilon}
d_X(f(x+\varepsilon),f(x))^2.
\]
Thus $\Kc(n,m)\le n<\infty$ for every even $m$.
\end{remark}

\section*{AI Declaration}

During the preparation of this work, the authors used GPT-6 Astra
to assist with language editing and the refinement of mathematical
arguments. The authors reviewed and revised the AI-assisted
material and independently checked the mathematical arguments
and references. The authors take full responsibility for the
accuracy and integrity of the final manuscript.

\end{document}